\documentclass[12pt]{article}
\usepackage{amsmath, amsthm, amscd, amsfonts, amssymb, graphicx}
\usepackage{titles}
\usepackage{algorithm}
\usepackage{algpseudocode}
\usepackage{subfigure} 
 \usepackage{caption}
 \usepackage{graphics,graphicx}

\usepackage{epsfig}
\usepackage{color}
\usepackage{fancyhdr}
\usepackage{titlesec}
\usepackage{tikz}
\usepackage{afterpage}
\usetikzlibrary{arrows}
\usetikzlibrary{decorations.markings}
\tikzstyle{block}=[draw opacity=0.7,line width=1.4cm]
\tikzset{
  big black arrow/.style={
    decoration={markings,mark=at position 1 with {\arrow[scale=2.5,black]{>}}},
    postaction={decorate},
    shorten >=0.4pt},
    line/.style={draw, ->}}

\usepackage{eufrak}

\setbox0=\hbox{$+$}
\newdimen\plusheight
\plusheight=\ht0
\def\+{\;\lower\plusheight\hbox{$+$}\;}

\setbox0=\hbox{$-$}
\newdimen\minusheight
\minusheight=\ht0
\def\-{\;\lower\minusheight\hbox{$-$}\;}

\setbox0=\hbox{$\cdots$}
\newdimen\cdotsheight
\cdotsheight=\plusheight
\def\cds{\lower\cdotsheight\hbox{$\cdots$}}

\renewcommand{\(}{\left\(}
\renewcommand{\)}{\right\)}

\renewcommand{\pmod}[1]{\,(\textup{mod}\,#1)}
\numberwithin{equation}{section}
\theoremstyle{plain}
\newtheorem{theorem}{Theorem}[section]
\newtheorem{lemma}[theorem]{Lemma}

\usepackage{pdfpages}

\begin{document}
\begin{center}{\bf Some New Congruences and Partition-Theoretic Interpretations for the Coefficients of Some Rogers-Ramanujan Type Identities
	}\end{center}
	\begin{center}	
	\footnotesize{\bf Sabi Biswas and Nipen Saikia$^{\ast}$}\\
					Department of Mathematics, Rajiv Gandhi
			University,\\ Rono Hills, Doimukh-791112, Arunachal Pradesh, India.\\
			E. Mail(s): sabi.biswas@rgu.ac.in; nipennak@yahoo.com\\
			$^\ast$\textit{Corresponding author}.\end{center}\vskip2mm
		
		\noindent {\bf Abstract:} Ramanujan listed several $q$-series identities in his lost notebook. The most well known $q$-series identities are the Rogers-Ramanujan type identities which are first discovered by Rogers and then rediscovered by Ramanujan. In this paper, we give partition-theoretic interpretations of some of the Rogers-Ramanujan type identities using overpartition and colour partition  of positive integers, and  prove infinite families of congruences modulo powers of $2$.
		\vskip 3mm
			\noindent  {\bf Keywords and phrases:} Rogers-Ramanujan type identities;  overpartition; colour partition; partition congruences. 
				\vskip 3mm
				\noindent  {\bf 2020 Mathematical Subject Classification:} 11P84; 11P83.
		
		\section{Introduction}
		For any complex numbers $A$ and $q$ with $|q|<1$ , a $q$-series is a summand containing the expression of the type
		
		$$(A;q)_{\infty} = \prod_{k=0}^{\infty}(1-Aq^k),\quad where\quad  (A;q)_{0} = 1,\quad (A;q)_{n} =\prod_{k=0}^{n-1}(1-Aq^k), \quad  n\ge1.$$
		
		For convenience, one often use the notation
		$$(A_{1};q)_{\infty}(A_{2};q)_{\infty}(A_{3};q)_{\infty}....(A_{k};q)_{\infty} = (A_{1},A_{2},A_{3},...,A_{k};q)_{\infty}.$$
		Throughout the paper, we write $\ell _{n} := (q^n;q^n)_{\infty},$ for any integer $n\geq1$. Ramanujan defined general theta-function $f(c,d)$ \cite[p. 34, (18.1)]{BBC} as
		\begin{equation}\label{eq3}
		f(c,d) = \sum_{m={-}\infty}^{\infty}c^{m(m+1)/2}d^{m(m-1)/2},  \quad                |cd|<1.
		\end{equation}The special cases \cite[p. 35, Entry 18]{BBC} of $f(c,d)$ are given by 
		\begin{equation}\label{eq4}
		\phi(q) := f(q,q) = \sum_{m={-}\infty}^{\infty}q^{{m}^2}=\frac{(-q;q^2)_\infty (q^2;q^2)_\infty}{(q;q^2)_\infty (-q^2;q^2)_\infty} = \frac{\ell_{2}^{5} }{{\ell_{1}^2}{\ell_{4}^2}} 
		\end{equation} and
		\begin{equation}\label{eq5}
		\psi(q) := f(q,q^3) = \sum_{m=0}^{\infty}q^{m(m+1)/2} = \frac{\ell_{2}^2}{\ell_{1}}.
		\end{equation}The product representations in the special cases \eqref{eq4}-\eqref{eq5} are the consequences of one of the celebrated result in the theory of $q$-series known as the Jacobi's triple product identity, given by
		\begin{equation}\label{jaco}f(c,d)= (-c;cd)_{\infty}(-d;cd)_{\infty}(cd;cd)_{\infty}.\end{equation}
	
		By using elementary $q$-operations, it is easily seen that
		\begin{equation}\label{eq6}
		\phi(-q)=\frac{(q;q)_\infty}{(-q;q)_\infty} = \frac{(q;q)_{\infty}^2}{(q ^2;q^2)_{\infty}} = \frac{\ell_{1}^2}{\ell_{2}}.
		\end{equation}
		
		Ramanujan~\cite{lostp2} listed several $q$-series identities in his lost notebook. The most well-known $q$-series identities are the Rogers-Ramanujan identities (RRI) given by
		\begin{equation}\label{eq1}
		\mathbb{S}_1(q):= \prod_{n=0}^{\infty}(1-q^{5n+1})^{-1}(1-q^{5n+4})^{-1}
		\end{equation} and \begin{equation}\label{eq2}
		\mathbb{S}_2(q):=\prod_{n=0}^{\infty}(1-q^{5n+2})^{-1}(1-q^{5n+3})^{-1}.
		\end{equation}
		The identities \eqref{eq1} and \eqref{eq2}  were first discovered by Rogers~\cite{RL} in $1893$ and then rediscovered by Ramanujan in $1913$. Partition-theoretic interpretations of \eqref{eq1} and \eqref{eq2} are given by MacMahon~\cite{MM}.
		
		 Recently, Afsharijoo~\cite{AP} established a recurrence relation which gives extended form of these identities where odd and even parts play different roles. Several Rogers-Ramanujan type identities (RRTIs) were also provided by Slater~\cite{SL} and Chu and Zhang~\cite{CW}. Gupta and Rana~\cite{RG1}  and Gupta et al. \cite{GRS} offered combinatorial interpretations of many RRTIs by using signed partitions which inspired them to explore more about signed partition and congruence properties of these identities. Gupta and Rana~\cite{GM1} collected severnteen RRTIs from~\cite{CW} and established some particular congruences modulo powers of $2, 3$ and $6$. 
		In this paper, we investigate following  RRTIs from \cite{CW} (also see \cite{GM1}) for their partition-theoretic interpretations and new congruence properties:  
		\begin{align}
		\label{gk}
		& G_k(q)=\sum_{n=0}^{\infty}g_{k}(n)q^n =\dfrac{(-q;q)_\infty}{(q;q)_\infty}\phi(q^k),\qquad k=2.\\
		\label{c1}
		& H(q)=\sum_{n=0}^{\infty}h(n)q^n=\sum_{n=0}^{\infty}\dfrac{(-q;q)_{2n}q^n}{(q;q)_{2n+1}}=\dfrac{(-q;q)_\infty}{(q;q)_\infty}(q^4,-q^4,-q^4;q^4)_\infty.
		\end{align}
		\begin{align}	
		\label{c7}
		& T(q)=\sum_{n=0}^{\infty}t(n)q^n=\sum_{n=0}^{\infty}\dfrac{(-q;q^2)_nq^n}{(q;q)_{2n+1}}=\dfrac{(-q;q)_\infty}{(q;q)_\infty}(q^{12},q^3,q^9;q^{12})_\infty.\\
		\label{c10}
		& M(q)=\sum_{n=0}^{\infty}m(n)q^n=\sum_{n=0}^{\infty}\dfrac{(-q;q)_{2n}q^n}{(q^2;q^2)_n}=\dfrac{(-q;q)_\infty}{(q;q)_\infty}(q^6,q,q^5;q^6)_\infty.\\
		\label{c12}
		& R(q)=\sum_{n=0}^{\infty}r(n)q^n=\sum_{n=0}^{\infty}\dfrac{(-q^2;q^2)_nq^{n(n+1)}}{(q;q)_{2n+1}}=\dfrac{(-q^2;q^2)_\infty}{(q^2;q^2)_\infty}(q^6,-q,-q^5;q^6)_\infty.\\
		\label{c14}
		& S(q)=\sum_{n=0}^{\infty}s(n)q^n=\sum_{n=0}^{\infty}\dfrac{(-1;q^2)_nq^{n(n+1)}}{(q;q)_{2n}}=\dfrac{(-q^2;q^2)_\infty}{(q^2;q^2)_\infty}(q^6,-q^3,-q^3;q^6)_\infty.
		\end{align}

		In Section 3, we offer partition-theoretic interpretation of the $q$-series identities \eqref{gk} and prove their congruence properties. In Section 4, we give partition-theoretic interpretations of \eqref{c1}. Some infinite families of congruence for the identities \eqref{c1}-\eqref{c14} modulo power of 2 are also proved. 
		
		To end the introduction, we define partition functions and their generating functions which are important in this paper.
		A partition of an positive integer $n$ can be defined as finite sequence of positive integers $(\beta_1, \beta_2, ..., \beta_k)$ such that $\sum_{j=1}^k\beta_j=n;\quad \beta_j\ge \beta_{j+1},$ where $\beta_j$ are called parts or summands of the partition. The number of partitions of $n$ is usually denoted by $p(n)$. As an illustration, $n=3$ has following three partitions: $3,\quad  2+1, \quad1+1+1.$
		
		For positive integer $n$, an overpartition of  $n$ is defined as the partition of $n$ in which the first occurence of each part may be overlined. If $O(n)$ denotes the number of overpartitions of $n$ then $$\sum_{n=0}^{\infty}O(n)q^n=\dfrac{(-q;q)_\infty}{(q;q)_\infty}.$$
		Ramanujan~\cite{BR} defined the general partition function $p_{t}(n)$ as
		$$\sum_{n = 0}^{\infty}p_{t}(n)q^n = \frac{1}{(q;q)_{\infty}^t}.$$
		For $t>0$,  $p_{t}(n)$ denoted the number of partition of $n$ where each part of the partition is assumed to have $t$ distinct colours.  Also, for positive integers $r$,$s$ and $t$, $$\dfrac{1}{(q^r;q^s)^t}$$ denotes the generating function of the number of partitions of a positive integer such that parts $\equiv r\pmod{s}$ has $t$ colours. 
		\section{Preliminaries}
		The following  lemmas will be used to prove our results. 
		
		\begin{lemma} (\cite[Theorem 2.1]{CG}). If p is an odd prime, then
		\begin{equation}\label{eq7}
		\psi(q) = \sum_{t=0}^{(p-3)/2} q^{(t^{2}+t)/2} f\left(q^{\left(p^{2}+(2t+1)p\right)/2}, q^{\left(p^{2}-(2t+1)p\right)/2}\right) + q^{(p^{2}-1)/8}\psi(q^{p^2}).
		\end{equation}
		Furthermore, $\frac{(t^{2}+t)}{2} \not\equiv \frac{(p^{2}-1)}{8} \pmod{p}\quad for\quad 0\leq t \leq (p-3)/2.$
		\end{lemma}
		
		\begin{lemma}(\cite[Theorem 2.2]{CG}). If $p\geq5$ is a prime, then
		\begin{multline}\label{eq8}
		\ell_1 = \sum_{\substack{t={-(p-1)/2} \\ t \ne {(\pm p-1)/6}}}^{(p-1)/2} (-1)^{t} q^{(3t^{2}+t)/2} f\left(-q^{3p^{2}+(6t+1)p/2}, -q^{3p^{2}-(6t+1)p/2}\right)\\+ (-1)^{(\pm p-1)/6} q^{(p^{2}-1)/24} \ell_{p^2},
		\end{multline} where
		\begin{equation*}
		\dfrac{\pm p-1}{6}
		= \left\{
		\begin{array}{cc}
		\dfrac{(p-1)}{6} \quad
		 if~ p \equiv 1\pmod{6}, \\
		\dfrac{(-p-1)}{6} \quad
		 if~ p \equiv -1\pmod{6}.
		\end{array}
		\right.
		\end{equation*} Furthermore, if $-\frac{p-1}{2} \leq t \leq \frac{p-1}{2} \quad and \quad t \ne \frac{\pm p-1}{6}, \quad then  \quad \frac{3t^{2}+t}{2} \not\equiv \frac{p^{2}-1}{24} \pmod{p}.$
		
		\end{lemma}
		
		\begin{lemma} We have,
		\begin{align}
		\label{eq12}
		&\dfrac{\ell_3^3}{\ell_1} = \dfrac{\ell_{4}^3 \ell_6^2}{\ell_2^2 \ell_{12}} + q\dfrac{\ell_{12}^3}{\ell_4},\\
		\label{eq14}
		&\dfrac{\ell_2^2}{\ell_1} = \dfrac{\ell_6 \ell_9^2}{\ell_3 \ell_{18}} + q\dfrac{\ell_{18}^2}{\ell_9},\\
		\label{c18} 
		&\dfrac{\ell_2}{\ell_1^2}=\dfrac{\ell_6^4 \ell_9^6}{\ell_3^8 \ell_{18}^3}+2q\dfrac{\ell_6^3 \ell_9^3}{\ell_3^7}+4q^2\dfrac{\ell_6^2 \ell_{18}^3}{\ell_3^6}.
		\end{align}
		\end{lemma}
 Identity \eqref{eq12} is Equation (22.1.14) in \cite{HD}. Identity \eqref{c18} can be found in \cite{HD}. Identity \eqref{eq14} is Equation (14.3.3) of \cite{HD}.

In addition to above identities, we need the following congruences which is easy consequence of the binomial theorem: 
  For positive integers $k$ and $m$, we have
\begin{align}\label{lm1}
&\ell_{2k}^{m} \equiv \ell_{k}^{2m}\pmod{2},\\
\label{lm2}
&\ell_{2k}^{2m} \equiv \ell_{k}^{4m}\pmod{4}.
\end{align}	 

In order to state our congruences, we will also use Legendre symbol which is defined as follows:

Let $p$ be any odd prime and $\xi$ be any integer relatively prime to $p$, then the Legendre symbol $\left(\dfrac{\xi}{p}\right)$ is defined by  
 	\begin{equation*}
 	\left(\dfrac{\xi}{p}\right)
 	= \left\{
 	\begin{array}{cc}
 	\hspace{-.5cm}1,~\text{if $\xi$ is quadratic residue modulo $p$}, \\
 	-1,~\text{if $\xi$ is quadratic non-residue modulo $p$} .
 	\end{array}
 	\right.
 	\end{equation*}

\section{Congruences for $g_k(n)$}
\begin{theorem}
If~ $g_k(n)$ is as defined in \eqref{gk}, then~ $g_k(n)$ is the number of overpartitions of a positive integer $n$ into parts such that no part is congruent to 0 modulo $2k$ and parts congruent to $k\pmod{2k}$ have two colours.
\end{theorem}
\begin{proof}
Employing \eqref{eq4} (with $q$ replaced by $q^k$) in right hand side of \eqref{gk}, we obtain
\begin{equation}\label{r1}
\sum_{n=0}^{\infty}g_k(n)q^n=\dfrac{(-q;q)_\infty (-q^k;q^{2k})_\infty (q^{2k};q^{2k})_\infty}{(q;q)_\infty (q^k;q^{2k})_\infty(-q^{2k};q^{2k})_\infty}.
\end{equation}The right hand side of \eqref{r1} is the generating function for the number of overpartitions of a positive integer $n$ into parts such that no part is congruent to 0 modulo $2k$ and parts congruent to $k\pmod{2k}$ have two colours. So, the proof is complete.
\end{proof}

\begin{theorem}\label{thm1}
For all integers $\alpha\geq0$,

$(i)$~Let $p\geq3$ be any prime and $1\leq j\leq (p-1)$, we have
\begin{equation}\label{a29}\hspace{-5.5cm}
\sum_{n=0}^{\infty} g_2\left(16\cdot p^{2\alpha}n + 2\cdot p^{2\alpha}\right)q^n \equiv 2\psi(q)\pmod{4},
\end{equation}
\begin{equation}\label{a30}\hspace{-4cm}
 g_2\left(16\cdot p^{2\alpha+2}n+16\cdot p^{2\alpha+1}j+2\cdot p^{{2\alpha}+2}\right) \equiv 0 \pmod{4}.
\end{equation}

$(ii)$~Let $p\geq5$ be any prime such that $\left(\dfrac{-8}{p}\right)=-1$ and $1\leq j\leq (p-1)$, we have
\begin{equation}\label{a31}
\hspace{-4.8cm}\sum_{n=0}^{\infty} g_2\left(8\cdot p^{2\alpha}n+3\cdot p^{2\alpha}\right)q^n\equiv 4\ell_1\ell_8\pmod{8},
\end{equation}
\begin{equation}\label{32}
\hspace{-3.5cm}g_2\left(8\cdot p^{2\alpha+2}n+8\cdot p^{2\alpha+1}j+3\cdot p^{2\alpha+2}\right)\equiv 0\pmod{8}.
\end{equation}
\end{theorem}
\begin{proof}
Setting $k=2$ in \eqref{gk} and simplifying using \eqref{eq4} (with $q$ replaced by $q^2$) and \eqref{jaco}, we obtain
 \begin{equation}\label{e14}
 \sum_{n=0}^{\infty}g_2(n)q^n = \dfrac{(-q;q)_{\infty}}{(q;q)_{\infty}}(q^4,-q^2,-q^2;q^4)_{\infty} = \dfrac{\ell_4^5}{\ell_1^2\ell_2\ell_8^2}.
 \end{equation}
(i)~From \cite[p. 9 Theorem 4]{GM1}, we note that
 \begin{equation}\label{a32}
 \sum_{n=0}^{\infty}g_2(8n+2)q^n \equiv 6\dfrac{\ell_2^7}{\ell_4^2}\pmod{4}.
 \end{equation}Extracting the terms involving $q^{2n}$ from \eqref{a32}, replacing $q^2$ by $q$ and using \eqref{lm1} and \eqref{eq5}, we obtain
 $$\sum_{n=0}^{\infty}g_2(16n+2)q^n \equiv2\psi(q)\pmod{4},$$which is the $\alpha=0$ case of \eqref{a29}. Assume that \eqref{a29} is true for some $\alpha\geq0$. Now, employing \eqref{eq7} in \eqref{a29} and extracting the terms involving $q^{pn+(p^{2}-1)/8}$, dividing by $q^{(p^{2}-1)/8}$ and replacing $q^p$ by $q$, we obtain
   \begin{equation}\label{g1}
   \sum_{n=0}^{\infty}g_2\left(16\cdot p^{2\alpha+1}n+2\cdot p^{2\alpha+2}\right)q^n\equiv2\psi(q^p)\pmod{4}.
   \end{equation} Extracting the terms involving $q^{pn}$ from \eqref{g1} and then replacing $q^p$ by $q$, we obtain
   $$\sum_{n=0}^{\infty}g_2\left(16\cdot p^{2\alpha+2}n+2\cdot p^{2\alpha+2}\right)q^n\equiv2\psi(q)\pmod{4},$$
  which is the $\alpha+1$ case of \eqref{a29}. Hence, by the method of induction, we complete the proof of \eqref{a29}. By extracting the terms involving $q^{pn+j}$, $1\leq j\leq(p-1)$ from \eqref{g1}, we arrive at \eqref{a30}.
 
 (ii)~From \cite[p. 10 Theorem 4]{GM1}, we note that
 \begin{equation}\label{a33}
  \sum_{n=0}^{\infty}g_2(4n+3)q^n \equiv 4\dfrac{\ell_4^{11}}{\ell_2^5\ell_8^2}\pmod{8}.
 \end{equation}Extracting the terms involving $q^{2n}$ from \eqref{a33}, replacing $q^2$ by $q$ and then using \eqref{lm1}, we obtain
 $$\sum_{n=0}^{\infty}g_2(8n+3)q^n \equiv4\ell_1\ell_8\pmod{8},$$which is the $\alpha=0$ case of \eqref{a31}. Assume that \eqref{a31} is true for some $\alpha\geq0$. Now, substituting \eqref{eq8} in \eqref{a31}, we obtain
 \begin{align}\label{g2}
 &\sum_{n=0}^{\infty} g_2\left(8\cdot p^{2\alpha}n+3\cdot p^{2\alpha}\right)q^n\notag\\ &\equiv4\Big[\sum_{\substack{t={-(p-1)/2} \\ t \ne {(\pm p-1)/6}}}^{(p-1)/2} (-1)^{t} q^{(3t^{2}+t)/2} f\left(-q^{\left(3p^{2}+(6t+1)p\right)/2}, -q^{\left(3p^{2}-(6t+1)p\right)/2}\right)\notag\\
 & \hspace{8cm}+ (-1)^{(\pm p-1)/6} q^{(p^{2}-1)/24} \ell_{p^2}\Big]\notag\\ &\times\Big[\sum_{\substack{m={-(p-1)/2} \\ m \ne {(\pm p-1)/6}}}^{(p-1)/2} (-1)^{m} q^{8(3m^{2}+m)/2} f\left(-q^{8\left(3p^{2}+(6m+1)p\right)/2}, -q^{8\left(3p^{2}-(6m+1)p\right)/2}\right)\notag\\
 &\hspace{8cm} + (-1)^{(\pm p-1)/6} q^{8(p^{2}-1)/24} \ell_{8p^2}\Big]\pmod{8}.
 \end{align}Consider, the congruence
 \begin{equation*}
 \dfrac{(3t^2+t)}{2}+8\left(\dfrac{m^2+m}{2}\right)\equiv 9\left(\dfrac{p^2-1}{24}\right)\pmod{p},
 \end{equation*}which is equivalent to \begin{equation}\label{g3}
 (6t+1)^2+8(6m+1)^2\equiv0\pmod{p}.
 \end{equation}For $\left(\dfrac{-8}{p}\right)=-1$, the congruence \eqref{g3} has only solution $t=m=(\pm p-1)/6$. Therefore, extracting the terms involving $q^{pn+9(p^{2}-1)/24}$ from \eqref{g2}, dividing by $q^{9(p^{2}-1)/24}$ and replacing $q^p$ by $q$, we obtain
 \begin{equation}\label{g4}
 \sum_{n=0}^{\infty} g_2\left(8\cdot p^{2\alpha+1}n+3\cdot p^{2\alpha+2}\right)q^n\equiv 4\ell_p\ell_{8p}\pmod{8}.
 \end{equation}Extracting the terms involving $q^{pn}$ from \eqref{g4} and then replacing $q^p$ by $q$, we obtain
 $$g_2\left(8\cdot p^{2\alpha+2}n+3\cdot p^{2\alpha+2}\right)q^n\equiv 4\ell_1\ell_8\pmod{8},$$ which is the $\alpha+1$ case of \eqref{a31}. Hence, by the method of induction, we complete the proof of \eqref{a31}. The result \eqref{32} follows from \eqref{g4} by extracting the terms involving $q^{pn+j}$, $1\leq j\leq(p-1)$.
\end{proof}

\section{Partition-theoretic interpretations and congruences for the identities \eqref{c1}-\eqref{c14}}
We first give partition-theoretic interpretations of \eqref{c1}. The partition-theoretic interpretations of \eqref{c7}, \eqref{c10} and \eqref{c12} can be found in \cite{SN}.
\begin{theorem}
If $h(n)$ is as defined in \eqref{c1}, then $h(n)$ is the number of partitions of a positive integer $n$ into parts such that no part $\equiv0\pmod{8}$, parts $\equiv 2,6\pmod{8}$ have one colour and parts $\equiv 1,3,4,5,7\pmod{8}$ have two colours.
\end{theorem}
\begin{proof}
 Simplifying right hand side of \eqref{c1}, we obtain
\begin{equation}\label{c2}
\sum_{n=0}^{\infty}h(n)q^n=\dfrac{(-q;q)_\infty}{(q;q)_\infty}(q^4,-q^4,-q^4;q^4)_\infty =\dfrac{(q^2;q^2)_\infty (q^8;q^8)_{\infty}^2}{(q;q)_{\infty}^2 (q^4;q^4)_{\infty}}.
\end{equation}
Changing the base $q$ in $(q^2;q^2)_\infty$, $(q^4;q^4)_\infty$ and $(q;q)_\infty$ to $q^8$, we obtain
\begin{align}\label{c3}
& (q^2;q^2)_\infty=(q^2,q^4,q^6,q^8;q^8)_\infty,\\
\label{c4}
& (q^4;q^4)_\infty=(q^4,q^8;q^8)_\infty,\\ 
\label{c5}
& (q;q)_\infty=(q,q^2,q^3,q^4,q^5,q^6,q^7,q^8;q^8)_\infty.
\end{align}
Employing \eqref{c3}-\eqref{c5} in \eqref{c2}, we arrive at our desired result.
\end{proof}
\begin{theorem}\label{thm4}For all integers $\alpha\geq0$,

$(i)$~Let $p\geq5$ be any prime such that $\left(\dfrac{-2}{p}\right)=-1$ and $1\leq j\leq (p-1)$, we have
\begin{equation}\label{a34}
\hspace{-5.2cm}\sum_{n=0}^{\infty} h\left(4\cdot p^{2\alpha}n+\dfrac{p^{2\alpha}-1}{2}\right)q^n\equiv \ell_1\ell_2\pmod{4},
\end{equation}
\begin{equation}\label{a35}
\hspace{-4cm}h\left(4\cdot p^{2\alpha+2}n+4\cdot p^{2\alpha+1}j+\dfrac{p^{2\alpha+2}-1}{2}\right)\equiv 0\pmod{4}.
\end{equation}

$(ii)$~Let $p\geq5$ be any prime such that $\left(\dfrac{-18}{p}\right)=-1$ and $1\leq j\leq (p-1)$, we have
\begin{equation}\label{a36}
\hspace{-3.8cm}\sum_{n=0}^{\infty} h\left(12\cdot p^{2\alpha}n+\dfrac{19\cdot p^{2\alpha}-1}{2}\right)q^n\equiv 2\ell_1\psi(q^6)\pmod{4},
\end{equation}
\begin{equation}\label{a37}
\hspace{-2.6cm}h\left(12\cdot p^{2\alpha+2}n+12\cdot p^{2\alpha+1}j+\dfrac{19\cdot p^{2\alpha+2}-1}{2}\right)\equiv 0\pmod{4}.
\end{equation}

$(iii)$~Let $p\geq5$ be any prime such that $\left(\dfrac{-2}{p}\right)=-1$ and $1\leq j\leq (p-1)$, we have
\begin{equation}\label{a38}
\hspace{-3.8cm}\sum_{n=0}^{\infty} h\left(12\cdot p^{2\alpha}n+\dfrac{11\cdot p^{2\alpha}-1}{2}\right)q^n\equiv 2\ell_2\psi(q^3)\pmod{4},
\end{equation}
\begin{equation}\label{a39}
\hspace{-2.6cm}h\left(12\cdot p^{2\alpha+2}n+12\cdot p^{2\alpha+1}j+\dfrac{11\cdot p^{2\alpha+2}-1}{2}\right)\equiv 0\pmod{4}.
\end{equation}

$(iv)$~Let $p\geq3$ be any prime such that $\left(\dfrac{-2}{p}\right)=-1$ and $1\leq j\leq (p-1)$ we have
\begin{equation}\label{2}
\hspace{-3cm}\sum_{n=0}^{\infty} h\left(4\cdot p^{2\alpha}n+\dfrac{3\cdot p^{2\alpha}-1}{2}\right)q^n\equiv 2\psi(q)\psi(q^2)\pmod{8},
\end{equation}
\begin{equation}\label{3}
\hspace{-2.8cm}h\left(4\cdot p^{2\alpha+2}n+4\cdot p^{2\alpha+1}j+\dfrac{3\cdot p^{2\alpha+2}-1}{2}\right)\equiv 0\pmod{8}.
\end{equation}
\end{theorem}
\begin{proof}
(i)~From \cite[p. 10 Theorem 5]{GM1}, we note that
\begin{equation}\label{a40}
\sum_{n=0}^{\infty}h(2n)q^n=\dfrac{\ell_4^7}{\ell_1^4\ell_2\ell_8^2}.
\end{equation}Using \eqref{lm2} in \eqref{a40} and then extracting the terms involving $q^{2n}$, replacing $q^2$ by $q$, we obtain
$$\sum_{n=0}^{\infty}h(4n)q^n\equiv\ell_1\ell_2\pmod{4},$$which is the $\alpha=0$ case of \eqref{a34}. Now, proceeding in the same way as in (ii) of Theorem~\ref{thm1}, we arrive at \eqref{a34} and \eqref{a35}.

(ii)~From \cite[p. 11 Theorem 5]{GM1}, we note that
\begin{equation}\label{a41}
\sum_{n=0}^{\infty}h(6n+3)q^n\equiv 2q\dfrac{\ell_4\ell_6^2\ell_{24}^2}{\ell_2\ell_{12}^2}\pmod{4}.
\end{equation}Extracting the terms involving $q^{2n+1}$ from \eqref{a41}, dividing by $q$, replacing $q^2$ by $q$ and then using \eqref{lm1}, we obtain
$$\sum_{n=0}^{\infty}h(12n+9)q^n\equiv2\ell_1\psi(q^6)\pmod{4},$$which is the $\alpha=0$ case of \eqref{a36}. Now, proceeding in the same way as in (ii) of Theorem~\ref{thm1}, we arrive at the results \eqref{a36} and \eqref{a37}.

(iii)~From \cite[p. 11 Theorem 5]{GM1}, we note that
\begin{equation}\label{a42}
\sum_{n=0}^{\infty}h(6n+5)q^n\equiv 2\dfrac{\ell_2^2\ell_{12}^2}{\ell_6}\pmod{4}.
\end{equation}Extracting the terms involving $q^{2n}$ from \eqref{a42}, replacing $q^2$ by $q$ and then using \eqref{lm1}, we obtain
$$\sum_{n=0}^{\infty}h(12n+5)q^n\equiv 2\ell_2\psi(q^3)\pmod{4},$$which is the $\alpha=0$ case of \eqref{a38}. Now, proceeding in the same way as in (ii) of Theorem~\ref{thm1}, we arrive at \eqref{a38} and \eqref{a39}.

(iv)~From \cite[p. 10 Theorem 5]{GM1}, we note that
\begin{equation}\label{a43}
\sum_{n=0}^{\infty}h(2n+1)q^n=2\dfrac{\ell_2\ell_4\ell_8^2}{\ell_1^4}.
\end{equation}Using \eqref{lm2} in \eqref{a43} and then extracting the terms involving $q^{2n}$, replacing $q^2$ by $q$, we obtain
$$\sum_{n=0}^{\infty}h(4n+1)q^n\equiv 2\dfrac{\ell_2\ell_4^2}{\ell_1}=2\dfrac{\ell_2^2\ell_4^2}{\ell_1\ell_2}=2\psi(q)\psi(q^2)\pmod{8},$$ which is the $\alpha=0$ case of \eqref{2}. Now, proceeding in the same way as in (ii) of Theorem~\ref{thm1}, we arrive at \eqref{2} and \eqref{3}.
\end{proof}
\begin{theorem}
For all integer $\alpha\geq0$,

$(i)$~Let $p\geq3$ be any prime and $1\leq j\leq (p-1)$, we have
\begin{equation}\label{4}
\hspace{-5.5cm}\sum_{n=0}^{\infty} t\left(3\cdot p^{2\alpha}n+\dfrac{3\cdot p^{2\alpha}-3}{8}\right)q^n\equiv \psi(q)\pmod{2},
\end{equation}
\begin{equation}\label{5}
\hspace{-4cm}t\left(3\cdot p^{2\alpha+2}n+3\cdot p^{2\alpha+1}j+\dfrac{3\cdot p^{2\alpha+2}-3}{8}\right)\equiv 0\pmod{2}.
\end{equation}

$(ii)$~Let $p\geq5$ be any prime such that $\left(\dfrac{-2}{p}\right)=-1$ and $1\leq j\leq (p-1)$, we have
\begin{equation}\label{6}
\hspace{-3.8cm}\sum_{n=0}^{\infty} t\left(3\cdot p^{2\alpha}n+\dfrac{11\cdot p^{2\alpha}-3}{8}\right)q^n\equiv 2\ell_2\psi(q^3)\pmod{4},
\end{equation}
\begin{equation}\label{7}
\hspace{-3cm}t\left(3\cdot p^{2\alpha+2}n+3\cdot p^{2\alpha+1}j+\dfrac{11\cdot p^{2\alpha+2}-3}{8}\right)\equiv 0\pmod{4}.
\end{equation}

$(iii)$~Let $p\geq5$ be any prime such that $\left(\dfrac{-18}{p}\right)=-1$ and $1\leq j\leq (p-1)$, we have
\begin{equation}\label{8}
\hspace{-3.5cm}\sum_{n=0}^{\infty} t\left(3\cdot p^{2\alpha}n+\dfrac{19\cdot p^{2\alpha}-3}{8}\right)q^n\equiv 4\ell_1\psi(q^6)\pmod{8},
\end{equation}
\begin{equation}\label{9}
\hspace{-2.6cm}t\left(3\cdot p^{2\alpha+2}n+3\cdot p^{2\alpha+1}j+\dfrac{19\cdot p^{2\alpha+2}-3}{8}\right)\equiv 0\pmod{8}.
\end{equation}
\end{theorem}
\begin{proof}
We have
\begin{equation}\label{10}
\sum_{n=0}^{\infty}t(n)q^n=\dfrac{(-q;q)_\infty}{(q;q)_\infty}(q^{12},q^3,q^9;q^{12})_\infty=\dfrac{\ell_2\ell_3\ell_{12}}{\ell_1^2\ell_6}.
\end{equation}Employing \eqref{c18} in \eqref{10}, we obtain
\begin{equation}\label{11}
\sum_{n=0}^{\infty}t(n)q^n=\dfrac{\ell_6^3\ell_9^6\ell_{12}}{\ell_3^7\ell_6^3}+2q\dfrac{\ell_6^2\ell_9^3\ell_{12}}{\ell_3^6}+4q^2\dfrac{\ell_6\ell_{12}\ell_{18}^3}{\ell_3^5}.
\end{equation}
(i)~Extracting the terms involving $q^{3n}$ from \eqref{11}, replacing by $q^3$ by $q$ and then using \eqref{lm1}, we obtain
$$ \sum_{n=0}^{\infty}t(3n)q^n\equiv\psi(q)\pmod{2},$$ which is the $\alpha=0$ case of \eqref{4}. Now, proceeding in the same way as in (i) of Theorem~\ref{thm1}, we arrive at \eqref{4} and \eqref{5}.

(ii)~Extracting the terms involving $q^{3n+1}$ from \eqref{11}, dividing by $q$, replacing by $q^3$ by $q$ and then using \eqref{eq5} and \eqref{lm1}, we obtain
$$\sum_{n=0}^{\infty}t(3n+1)q^n\equiv2\ell_2\psi(q^3)\pmod{4},$$ which is the $\alpha=0$ case of \eqref{6}. Now, proceeding in the same way as in (ii) of Theorem~\ref{thm1}, we arrive at \eqref{6} and \eqref{7}.

(iii)~Extracting the terms involving $q^{3n+2}$ from \eqref{11}, dividing by $q^2$, replacing by $q^3$ by $q$ and then using \eqref{eq5} and \eqref{lm1}, we obtain
$$\sum_{n=0}^{\infty}t(3n+2)q^n\equiv4\ell_1\psi(q^6)\pmod{8},$$ which is the $\alpha=0$ case of \eqref{8}. Now, proceeding in the same way as in (ii) of Theorem~\ref{thm1}, we arrive at \eqref{8} and \eqref{9}.
\end{proof}

\begin{theorem}For all integer $\alpha\geq0$,

$(i)$~Let $p\geq5$ be any prime and $1\leq j\leq (p-1)$, we have
\begin{equation}\label{12}\hspace{-6cm}
 m\left(8\cdot p^{2}n + 8\cdot pj+\dfrac{p^2+2}{3}\right)\equiv 0\pmod{2}.
\end{equation}

$(ii)$~Let $p\geq5$ be any prime such that $\left(\dfrac{-1}{p}\right)=-1$ and $1\leq j\leq (p-1)$, we have
\begin{equation}\label{13}
\hspace{-4.5cm}\sum_{n=0}^{\infty} m\left(16\cdot p^{2\alpha}n+\dfrac{10\cdot p^{2\alpha}-1}{3}\right)q^n\equiv 4\ell_1\ell_4\pmod{8},
\end{equation}
\begin{equation}\label{14}
\hspace{-3cm}m\left(16\cdot p^{2\alpha+2}n+16\cdot p^{2\alpha+1}j+\dfrac{10\cdot p^{2\alpha+2}-1}{3}\right)\equiv 0\pmod{8}.
\end{equation}
\end{theorem}
\begin{proof}
(i)~From \cite[p.20 Theorem 15]{GM1}, we note that
\begin{equation}\label{15}
\sum_{n=0}^{\infty} m(2n+1)q^n\equiv\dfrac{\ell_4^2\ell_6^4}{\ell_2^2\ell_{12}^2}\pmod{2}.
\end{equation}Using \eqref{lm1} and then extracting the terms involving $q^{4n}$, replacing $q^4$ by $q$, we obtain
\begin{equation}\label{m1}
\sum_{n=0}^{\infty} m(8n+1)q^n\equiv\ell_1\pmod{2}.
\end{equation}
Substituting \eqref{eq8} in \eqref{m1} and then extracting the terms involving $q^{pn+(p^2-1)/24}$, dividing by $q^{(p^2-1)/24}$ and then replacing $q^p$ by $q$, we obtain
\begin{equation}\label{m2}
\sum_{n=0}^{\infty}m\left(8\cdot pn+\dfrac{p^2+2}{3}\right)q^n\equiv (-1)^{(\pm p-1)/6}\ell_p\pmod{2}.
\end{equation}Hence, the result easily follows from \eqref{m2} by extracting the terms involving $q^{pn+j}$, $1\leq j\leq(p-1)$.

(ii)~From \cite[p.21 Theorem 15]{GM1}, we note that
\begin{equation}\label{16}
\sum_{n=0}^{\infty} m(8n+3)q^n\equiv4\dfrac{\ell_2\ell_4^2\ell_6^4}{\ell_{12}^2}\pmod{8}.
\end{equation}Extracting the terms involving $q^{2n}$ from \eqref{16}, replacing $q^2$ by $q$ and then using \eqref{lm1}, we obtain
\begin{equation*}
\sum_{n=0}^{\infty} m(16 n+3)q^n\equiv4\ell_1\ell_4\pmod{8},
\end{equation*}
which is the $\alpha=0$ case of \eqref{13}. Now, proceeding in the same way as in (ii) of Theorem~\ref{thm1}, we arrive at \eqref{13} and \eqref{14}.
\end{proof}

\begin{theorem}For all integer $\alpha\geq0$,

$(i)$~Let $p\geq5$ be any prime such that $\left(\dfrac{-6}{p}\right)=-1$ and $1\leq j\leq (p-1)$, we have
\begin{equation}\label{17}
\hspace{-4cm}\sum_{n=0}^{\infty} r\left(8\cdot p^{2\alpha}n+\dfrac{7\cdot p^{2\alpha}-1}{3}\right)q^n\equiv 2\ell_1\psi(q^2)\pmod{4},
\end{equation}
\begin{equation}\label{18}
\hspace{-3.5cm}r\left(8\cdot p^{2\alpha+2}n+8\cdot p^{2\alpha+1}j+\dfrac{7\cdot p^{2\alpha+2}-1}{3}\right)\equiv 0\pmod{4}.
\end{equation}

$(ii)$~Let $p\geq5$ be any prime such that $\left(\dfrac{-1}{p}\right)=-1$ and $1\leq j\leq (p-1)$, we have
\begin{equation}\label{19}
\hspace{-4.5cm}\sum_{n=0}^{\infty} r\left(16\cdot p^{2\alpha}n+\dfrac{10\cdot p^{2\alpha}-1}{3}\right)q^n\equiv 2\ell_1\ell_4\pmod{4},
\end{equation}
\begin{equation}\label{20}
\hspace{-3cm}r\left(16\cdot p^{2\alpha+2}n+16\cdot p^{2\alpha+1}j+\dfrac{10\cdot p^{2\alpha+2}-1}{3}\right)\equiv 0\pmod{4}.
\end{equation}

$(iii)$~Let $p\geq5$ be any prime such that $\left(\dfrac{-2}{p}\right)=-1$ and $1\leq j\leq (p-1)$, we have
\begin{equation}\label{21}
\hspace{-3.8cm}\sum_{n=0}^{\infty} r\left(16\cdot p^{2\alpha}n+\dfrac{22\cdot p^{2\alpha}-1}{3}\right)q^n\equiv 2\ell_2\psi(q^3)\pmod{4},
\end{equation}
\begin{equation}\label{22}
\hspace{-3cm}r\left(16\cdot p^{2\alpha+2}n+16\cdot p^{2\alpha+1}j+\dfrac{22\cdot p^{2\alpha+2}-1}{3}\right)\equiv 0\pmod{4}.
\end{equation}
\end{theorem}
\begin{proof}
(i)~From \cite[p.21 Theorem 16]{GM1}, we note that
\begin{equation}\label{23}
\sum_{n=0}^{\infty} r(4n+2)q^n=2\dfrac{\ell_2\ell_6^2\ell_8^2}{\ell_1^4\ell_{12}}.
\end{equation}Using \eqref{lm1} in \eqref{23} and then extracting the terms involving $q^{2n}$, replacing $q^2$ by $q$, we obtain
$$\sum_{n=0}^{\infty} r(8n+2)q^n\equiv2\ell_1\psi(q^2)\pmod{4},$$which is the $\alpha=0$ case of \eqref{17}. Now, proceeding in the same way as in (ii) of Theorem~\ref{thm1}, we arrive at \eqref{17} and \eqref{18}.

(ii)~From \cite[p.22 Theorem 16]{GM1}, we note that
\begin{equation}\label{24}
\sum_{n=0}^{\infty} r(8n+3)q^n\equiv 2\dfrac{\ell_4\ell_8\ell_{12}^2}{\ell_2\ell_{24}}\pmod{8}.
\end{equation}Using \eqref{lm1} in \eqref{24} and then extracting the terms involving $q^{2n}$, replacing $q^2$ by $q$, we obtain
$$\sum_{n=0}^{\infty} r(16n+3)q^n\equiv 2\ell_1\ell_4\pmod{4},$$which is the $\alpha=0$ case of \eqref{19}. Now, proceeding in the same way as in (ii) of Theorem~\ref{thm1}, we arrive at \eqref{19} and \eqref{20}.

(iii)~From \cite[p.22 Theorem 16]{GM1}, we note that
\begin{equation}\label{25}
\sum_{n=0}^{\infty} r(8n+7)q^n\equiv 2\dfrac{\ell_4^4\ell_6\ell_{24}}{\ell_2^2\ell_8\ell_{12}}\pmod{8}.
\end{equation}Using \eqref{eq5} and \eqref{lm1} in \eqref{25} and then extracting the terms involving $q^{2n}$, replacing $q^2$ by $q$, we obtain
$$\sum_{n=0}^{\infty} r(16n+7)q^n\equiv 2\ell_2\psi(q^3)\pmod{4},$$which is the $\alpha=0$ case of \eqref{21}. Now, proceeding in the same way as in (ii) of Theorem~\ref{thm1}, we arrive at \eqref{21} and \eqref{22}.
\end{proof}

\begin{theorem}We have
\begin{align}
\label{26}
& s(24n+i)\equiv0\pmod{4},\quad where \quad i=9,15,21,\\
\label{27}
& s(24n+23)\equiv 0\pmod{8},\\
\label{28}
& s(24n+17)\equiv 0\pmod{8},\\
\label{29}
& s(12n+1)\equiv 0\pmod{16}.
\end{align}
\end{theorem}
\begin{proof}
~From \cite[p.12 Theorem 7]{GM1}, we note that
\begin{equation}\label{30}
\sum_{n=0}^{\infty}s(2n+1)q^n=2q\dfrac{\ell_2\ell_{24}^2}{\ell_1^2\ell_{12}}.
\end{equation}Employing \eqref{c18} in \eqref{30}, we obtain
\begin{equation}\label{31}
\sum_{n=0}^{\infty}s(2n+1)q^n=2q\dfrac{\ell_6^4\ell_9^6\ell_{24}^2}{\ell_3^8\ell_{12}\ell_{18}^3}+4q^2\dfrac{\ell_6^3\ell_9^3\ell_{24}^2}{\ell_3^7\ell_{12}}+8q^3\dfrac{\ell_6^2\ell_{18}^3\ell_{24}^2}{\ell_3^6\ell_{12}}.
\end{equation}Extracting the terms involving $q^{3n}$ from \eqref{31}, replacing $q^3$ by $q$ and then using \eqref{lm1}, we obtain
\begin{equation}\label{33}
\sum_{n=0}^{\infty}s(6n+1)q^n\equiv 8q\dfrac{\ell_6^3\ell_8^2}{\ell_2\ell_4}\pmod{16}.
\end{equation}Hence, the result \eqref{29} easily follows from \eqref{33} by extracting the terms involving $q^{2n}$. Now, extracting the terms involving $q^{3n+2}$ from \eqref{31}, dividing by $q^2$, replacing $q^3$ by $q$, using \eqref{lm1} and then employing \eqref{eq12} and \eqref{eq14}, we obtain 
\begin{equation}\label{34}
\sum_{n=0}^{\infty}s(6n+5)q^n\equiv 4\dfrac{\ell_4^2\ell_6^2\ell_8^2}{\ell_2^2\ell_{12}}+4q\dfrac{\ell_8^2\ell_{12}^3}{\ell_4^2}\pmod{8}.
\end{equation}Extracting the terms involving $q^{2n}$ from \eqref{34}, replacing $q^2$ by $q$ and then using \eqref{lm1}, we obtain
\begin{equation}\label{35}
\sum_{n=0}^{\infty}s(12n+5)q^n\equiv4\ell_2\ell_8\pmod{8}.
\end{equation}Hence, the result \eqref{28} easily follows from \eqref{35} by extracting the terms involving $q^{2n+1}$. Now, extracting the terms involving $q^{2n+1}$ from \eqref{34}, dividing by $q$, replacing $q^2$ by $q$, we obtain
\begin{equation}\label{s1}
\sum_{n=0}^{\infty}s(12n+11)q^n\equiv 4\dfrac{\ell_4^2\ell_6^3}{\ell_2^2}\pmod{8}.
\end{equation}Hence, the result \eqref{27} easily follows from \eqref{s1} by extracting the terms involving $q^{2n+1}$.

Now, extracting the terms involving $q^{3n+1}$ from \eqref{31}, dividing by $q$, replacing $q^3$ by $q$ and then using \eqref{eq5} and \eqref{lm1}, we obtain 
\begin{equation}\label{36}
\sum_{n=0}^{\infty}s(6n+3)q^n\equiv 2\psi(q^4)\pmod{4}.
\end{equation}Hence, the result \eqref{26} easily follows from \eqref{36} by extracting the terms involving $q^{4n+i}$, $i=1,2,3$.
\end{proof}

\begin{theorem}For all integer $\alpha\geq0$,

$(i)$~Let $p\geq5$ be any prime such that $\left(\dfrac{-1}{p}\right)=-1$ and $1\leq j\leq (p-1)$, we have
\begin{equation}\label{39}
\hspace{-4.5cm}\sum_{n=0}^{\infty} s\left(24\cdot p^{2\alpha}n+3\cdot p^{2\alpha}\right)q^n\equiv 4\ell_1\ell_4\pmod{8},
\end{equation}
\begin{equation}\label{40}
\hspace{-3cm}s\left(24\cdot p^{2\alpha+2}n+24\cdot p^{2\alpha+1}j+3\cdot p^{2\alpha+2}\right)\equiv 0\pmod{8}.
\end{equation}

$(ii)$~Let $p\geq5$ be any prime such that $\left(\dfrac{-2}{p}\right)=-1$ and $1\leq j\leq (p-1)$, we have
\begin{equation}\label{41}
\hspace{-3.8cm}\sum_{n=0}^{\infty} s\left(24\cdot p^{2\alpha}n+11\cdot p^{2\alpha}\right)q^n\equiv 4\ell_2\psi(q^3)\pmod{8},
\end{equation}
\begin{equation}\label{42}
\hspace{-3cm}s\left(24\cdot p^{2\alpha+2}n+24\cdot p^{2\alpha+1}j+11\cdot p^{2\alpha+2}\right)\equiv 0\pmod{8}.
\end{equation}

$(iii)$~Let $p\geq5$ be any prime such that $\left(\dfrac{-6}{p}\right)=-1$ and $1\leq j\leq (p-1)$, we have
\begin{equation}\label{43}
\hspace{-4cm}\sum_{n=0}^{\infty} s\left(24\cdot p^{2\alpha}n+7\cdot p^{2\alpha}\right)q^n\equiv 8\ell_1\psi(q^2)\pmod{16},
\end{equation}
\begin{equation}\label{44}
\hspace{-3.5cm}s\left(24\cdot p^{2\alpha+2}n+24\cdot p^{2\alpha+1}j+7\cdot p^{2\alpha+2}\right)\equiv 0\pmod{16}.
\end{equation}

$(iv)$~Let $p\geq5$ be any prime such that $\left(\dfrac{-18}{p}\right)=-1$ and $1\leq j\leq (p-1)$, we have
\begin{equation}\label{45}
\hspace{-3.5cm}\sum_{n=0}^{\infty} s\left(24\cdot p^{2\alpha}n+19\cdot p^{2\alpha}\right)q^n\equiv 8\ell_1\psi(q^6)\pmod{16},
\end{equation}
\begin{equation}\label{46}
\hspace{-2.6cm}s\left(24\cdot p^{2\alpha+2}n+24\cdot p^{2\alpha+1}j+19\cdot p^{2\alpha+2}\right)\equiv 0\pmod{16}.
\end{equation}
\end{theorem}

\begin{proof}
(i)~Extracting the terms involving $q^{2n}$ from \eqref{35} and then replacing $q^2$ by $q$, we obtain
$$\sum_{n=0}^{\infty}s(24n+5)q^n\equiv 4\ell_1\ell_4\pmod{8},$$ which is the $\alpha=0$ case of \eqref{39}. Now, proceeding in the same way as in (ii) of Theorem~\ref{thm1}, we arrive at \eqref{39} and \eqref{40}.

(ii)~Extracting the terms involving $q^{2n}$ from \eqref{s1}, replacing $q^2$ by $q$ and then using \eqref{eq5} and \eqref{lm1}, we obtain
$$\sum_{n=0}^{\infty}s(24n+11)q^n\equiv 4\ell_2\psi(q^3)\pmod{8},$$ which is the $\alpha=0$ case of \eqref{41}. Now, proceeding in the same way as in (ii) of Theorem~\ref{thm1}, we arrive at \eqref{41} and \eqref{42}.

(iii)~Extracting the terms involving $q^{2n+1}$ from \eqref{33}, dividing by $q$, replacing $q^2$ by $q$ and then employing \eqref{eq12}, we obtain 
\begin{equation}\label{47}
\sum_{n=0}^{\infty}s(12n+7)q^n\equiv 8\dfrac{\ell_4^5\ell_6^2}{\ell_2^3\ell_{12}}+8q\dfrac{\ell_4\ell_{12}^3}{\ell_2}\pmod{16}.
\end{equation}Extracting the terms involving $q^{2n}$ from \eqref{47}, replacing $q^2$ by $q$ and then using \eqref{eq5} and \eqref{lm1}, we obtain
$$\sum_{n=0}^{\infty}s(24n+7)q^n\equiv 8\ell_1\psi(q^2)\pmod{16},$$ which is the $\alpha=0$ case of \eqref{43}. Now, proceeding in the same way as in (ii) of Theorem~\ref{thm1}, we arrive at \eqref{43} and \eqref{44}.

(iv)~Extracting the terms involving $q^{2n+1}$ from \eqref{47}, dividing by $q$, replacing $q^2$ by $q$ and then using \eqref{eq5} and \eqref{lm1}, we obtain 
$$\sum_{n=0}^{\infty}s(24n+19)q^n\equiv 8\ell_1\psi(q^6)\pmod{16},$$ which is the $\alpha=0$ case of \eqref{45}. Now, proceeding in the same way as in (ii) of Theorem~\ref{thm1}, we arrive at \eqref{45} and \eqref{46}.
\end{proof}

\section*{Acknowledgements} The first author acknowledge the financial support received from UGC, India through National Fellowship for Scheduled Caste  Students (NFSC) under grant Ref. no.: 211610029643.
	
	\section*{\bf Declarations}
	\noindent{\bf Conflict of Interest.} The authors declare that there is no conflict of interest regarding the publication of
	this article.
	
	\noindent{\bf Human and animal rights.} The authors declare that there is no research involving human participants or
	animals in the contained of this paper.	
	
	\noindent{\bf Data availability statements.} Data sharing not applicable to this article as no datasets were generated or analysed during the current study.

	\end{document}